\documentclass[12pt]{amsart}
\usepackage{enumerate}
\usepackage{amssymb}
\usepackage{graphicx}
\usepackage{amscd, color}
\usepackage{amsmath}
\usepackage{amsthm}
\usepackage{amsfonts}
\usepackage[english]{babel}
\usepackage{epsfig}
\usepackage{color}
\usepackage{mathtools}
\usepackage{tikz-cd}

\numberwithin{equation}{section}

\newtheorem{theorem}{Theorem}[section]

\newtheorem{proposition}[theorem]{Proposition}

\newtheorem{example}[theorem]{Example}

\newtheorem{remark}[theorem]{Remark}

\newtheorem{lemma}[theorem]{Lemma}
\newtheorem{question}[theorem]{Question}
\newtheorem{definition}[theorem]{Definition}
\newtheorem{claim}{Claim}
\usepackage{tikz}
\usepackage{hyperref}
\usetikzlibrary{matrix}

\DeclareMathOperator{\ho}{Hom}

\DeclareMathOperator{\cl}{CL}

\begin{document}

{\it This paper has been accepted for publication in Topology and its Applications.}

\vspace{.5cm}

{\it Dedicated to Sergey Antonyan on the occasion of his 65th birthday.}

\vspace{8mm}

\title[Some notes on induced functions]{Some notes on induced functions and group actions on hyperspaces}

\author[V. DOnju\'an, N. Jonard-P\'erez and A. L\'opez-Poo]{Victor Donju\'an, Natalia Jonard-P\'erez and Ananda L\'opez-Poo}

\subjclass[2010]{54B20, 54C05, 54C35,  54H15, 57S05}

\keywords{Hyperspaces, Attouch-Wets Metric, Fell topology, Vietoris topology, Induced functions, Group action, Group topology}

\thanks{The first author has been supported by CONACyT grant 464720 (M\'exico). The second  and third authors have been supported by CONACyT grant 252849 (M\'exico) and by  PAPIIT grant IN115819  (UNAM, M\'exico).}

\address{Departamento de  Matem\'aticas,
Facultad de Ciencias, Universidad Nacional Aut\'onoma de M\'exico, 04510 Ciudad de M\'exico, M\'exico.}
\email{(V. Donju\'an) donjuan@ciencias.unam.mx}
\email{(N.\,Jonard-P\'erez) nat@ciencias.unam.mx}
\email{(A. L\'opez-Poo) ananda@ciencias.unam.mx}

\begin{abstract}

Let $X$ be a topological space and $\cl(X)$ be the family of all nonempty closed subsets of $X$.

In this paper we discuss the problem of when a continuous map between topological spaces induces a continuous function between their respective hyperspaces. 
As a main result we characterize the continuity of the induced function in the case of the Fell and Attouch-Wets hyperspaces.
Additionally we explore the problem of whether a continuous action of a topological group $G$ on a topological space $X$ induces a continuous action on $\cl(X)$. In particular we give sufficient conditions on the topology of $G$ to guarantee that the induced action on $\cl(X)$ is continuous, provided that $\cl(X)$ is equipped with the Hausdorff or the Attouch-Wets metric topology.  
\end{abstract}

\maketitle

\section{Introduction}

 In \cite{West}, J. West considered  hyperspaces of a group $G$, equipped with the natural action induced by the group. Since that paper, several results appear in the literature concerning group actions on hyperspaces (see e. g. \cite{S1, S2, S3,  natalia sergey, SNS, Jonard 2014, Jonard accepted}). Some older results such as \cite{Grumbaum, Kuchment, Kuchment E} also considered group actions on hyperspaces of convex bodies, but from a convex geometry perspective. 
 However, in the majority of the existing literature,  only group actions on  hyperspaces of compact sets (equipped with the topology induced by the Hausdorff metric) have been considered.  The reason behind this choice could be that in this case, the induced action is always continuous. Nevertheless, if we change the hyperspace or the given topology, then  the continuity of the induced action may fail.

Notice that in order to explore whether a group $G$ induces a continuous action on the hyperspace $\cl(X)$ of a $G$-space $X$,  we first need to solve a simpler question:  whether a continuous  map $f:X\to Y$ induces a continuous function $\widetilde f:\cl(X)\to \cl (Y)$ (see equation~(\ref{e:induced function}) for the precise definition of $\widetilde f$).   
Let us observe that in some contexts (such as dynamical systems \cite{banks, wang}) the induced function plays an important role, and therefore this problem has been long studied (see e.g. \cite{Hitchcock,michael, wang}). Until now, the continuity of $\widetilde f$ is clear if we endow the hyperspace $\cl(X)$ with the Vietoris or the Hausdorff metric topology. However, for other hypertopologies such as the Fell topology, there are only partial results. And, as far as we know, the case of the Attouch-Wets metric topology has not  even been approached yet. 

The aim of this paper is to discuss the problem of whether the induced function or the induced action becomes continuous. We will concentrate in four hypertopologies: the Vietoris topology, the Fell topology, the Hausdorff metric topology and the Attouch-Wets topology. In the first three cases, we will also consider their upper and lower versions (see Section~\ref{s: hyperspaces} for a more precise definition).

As a main result we will characterize when a continuous function $f:X\to Y$ between metric spaces induces a continuous function between the hyperspaces $\cl(X)$ and $\cl(Y)$, provided that these hyperspaces are  equipped with the Attouch-Wets metric topology.
We also characterize the case when $f$ is closed and $\cl(X)$ and $\cl(Y)$ are equipped with the Fell topology. 
With respect to the induced action, we will give sufficient conditions in order to guarantee that  if $G$ acts continuously on a metric space $X$ then the induced action on $\cl(X)$, equipped with the Hausdorff metric topology or the Attouch metric topology, is continuous.

The paper is organized as follows. In Section~\ref{s: hyperspaces} we recall the definition of the four hypertopologies that we will be working with and how  they relate to each other. 
  
  In Section \ref{s:induced function} we will concentrate on the problem of the induced function. After a brief remainder on what is  known about it, we will prove our results related with the induced function on hyperspaces equipped with the Fell and the Attouch-Wets topologies. 
  
  The problem of the induced action will be approached in Section~\ref{s:acciones}. We will start by presenting some easy remarks related with actions on hyperspaces equipped with the Fell and the Vietoris topologies. Then we will discuss the problem of the continuity of the induced action on the hyperspaces of closed  and  closed and bounded subsets of a metric space, equipped with the Hausdorff  and the Attouch-Wets metric topologies.

\section{Hyperspaces}\label{s: hyperspaces}

Unless we say otherwise, throughout the paper all topological spaces $X$  are assumed to be Hausdorff.

The family of all nonempty closed subsets of a topological space $X$ will be denoted by $\cl(X)$. On the other hand, let us denote by $\mathcal K(X)$ the family of all nonempty compact subsets of $X$. Since all our spaces are assumed to be Hausdorff, we have that $\mathcal K(X)\subset \cl(X)$.

\subsection{Vietoris and Fell topologies}

Let $Q$ be any subset of $X$ and let us consider
$$Q^-=\{A\in\cl(X): A\cap Q\neq\emptyset\}$$
$$Q^+=\{A\in\cl(X): A\subset Q\}.$$

\begin{definition}\label{d:vietoris}
\begin{enumerate}
    \item The lower Vietoris topology on $\cl(X)$, denoted by $\tau_{V^-}$,
     is the topology generated by all sets of the form $U^-$, where $U$ is an open subset of $X$. 
    \item The upper Vietoris topology on $\cl(X)$, denoted by $\tau_{V^+}$, is the topology generated by all sets of the form $U^+$, where $U$ is an open subset of $X$.
    \item The Vietoris topology on $\cl(X)$, denoted by $\tau_V$, is the topology on  $\cl(X)$ generated by $\tau_{V^-}$ and $\tau_{V^+}$. 
\end{enumerate}

\end{definition}

We use the notation 
    $\cl_{V^-}(X)$ to denote the hyperspace $(\cl(X), \tau_{V^-})$. The subset  $\mathcal K(X)$
    of $\cl_{V^{-}}(X)$, equipped with the subspace topology, will be denoted by  $\mathcal K_{V^-}(X)$.
    Following the same logic, we define the hyperspaces 
 $\cl_{V^+}(X)$, $\cl_{V}(X)$, $\mathcal K_{V^+}(X)$ and $\mathcal K_{V}(X)$.

\begin{definition}\label{d:Fell}
\begin{enumerate}
     \item The upper Fell topology on $\cl(X)$, denoted by $\tau_{F^+}$, is the topology generated by all sets of the form $(X\setminus K)^+$, where $K$ is a compact subset of $X$.
     \item The Fell topology on $\cl(X)$,  is the topology generated by $\tau_{V^-}$ and $\tau_{F^+}$
\end{enumerate}
\end{definition}

As we do with the Vietoris topology, we denote by 
   $\cl_{F^+}(X)$ and $\cl_{F}(X)$ the hyperspaces  $(\cl(X),\tau_{F^{+}})$ and $(\cl(X),\tau_{F})$, respectively. Similarly we define $\mathcal K_{F^+}(X)$ and $\mathcal K_{F}(X)$.

Clearly, if $X$ is a Hausdorff space, the Fell topology is weaker than the Vietoris topology. 
However, in general these topologies are not the same, not even on $\mathcal K(X)$.

It is well known that $\cl_V(X)$ and $\cl_{F}(X)$ are $T_1$-spaces, provided that $X$ is a $T_1$-space (see e.g. \cite{michael}). However, in general, the hyperspaces $\cl_{V^-}(X)$, $\cl_{V^+}(X)$ and $\cl_{F^+}(X)$  are never  $T_1$, no matter what good topological properties the space $X$ has.  Yet, $\cl_{V^-}(X)$ and $\cl_{F^+}(X)$ are always $T_0$, even if $X$ does not satisfy any separation axiom, while  $\cl_{V^+}(X)$ is $T_0$, provided that  $X$ is a $T_1$-space.

Even if these weaker topologies may seem very poor, if we  restrict them to a specific family of closed sets,  they can become very useful and rich.
For example, if we consider the family of all $k$-dimensional subspaces of $\mathbb R^n$ equipped with the lower Vietoris topology, we obtain a compact smooth manifold: the Grassmann manifold $Gr(k, \mathbb R^n)$   (see e.g.            \cite{Resende-Santos}).
On the other hand, the hyperspace of all closed and convex subsets of a Banach space equipped with the lower Vietoris topology appears in many selection theorems, such as the classical Michael's Selection Theorem. The upper Vietoris topology also appears in several selection theorems (see, e.g.\cite{Repovs-Semenov}).
After all these examples, we found it was interesting to approach our goal by taking into account the role played  by each of these weaker topologies.

\subsection{Metric topologies}\label{ss:metric topologies}

If  $(X,d)$ is a metric space, we denote by $B(x,r)$ the open ball with center $x\in X$ and radius $r>0$.  Namely, $B(x,r)=\{y\in X\mid d(y,x)<r\}$. In this case, for any subset $A\subset X$, we denote by $N_r(A)$ the set of points $r$-closed to $A$, i.e;
$$N_r(A)=\{x\in X\mid d(x,A)<r\}.$$

For  a metric space  $(X,d)$, there are two well-known metrics on $\cl(X)$: the Hausdorff metric and the Attouch-Wets metric. 
For any $A,B\subset X$, consider the value
$$e(A, B)=\sup\{d(a, B)\mid a\in A\}=\inf\{t>0\mid A\subset N_t(B)\}.$$

The value $e(A,B)$ is called the excess of $A$ over $B$. It is not difficult to see that $e(A,B)=e(\overline{A},\overline{B})$ for any $A,B\subset X$, and it defines a quasipseudometric $e:\cl(X)\times\cl(X)\to [0,\infty]$ (see \cite[Section 1.5]{beer book}). 
The Hausdorff metric on $\cl(X)$ is defined 
  for every$A,B\in \cl(X)$, as  $$d_H(A,B)=\max\left\{e(A,B), e(B,A)\right\}.$$

On the other hand, we define the Attouch-Wets metric on $\cl(X)$ as  in~\cite{Sakai-Yaguchi}:
\begin{equation}\label{e:AW}
    d_{AW}(A,B)=\sup_{j\in\mathbb{N}}\left\{ \min\left\{{\frac 1j, \sup_{d(x_0,x)<j}|d(x,A)-d(x,B)|}\right\}\right\}
\end{equation}
where $x_0\in X$ is fixed. The topology generated by $d_{AW}$ does not depend on the point $x_0$ and this metric is equivalent to the Attouch-Wets metric  defined in \cite[Definition 3.1.2]{beer book}.
The reason why we prefer to use $d_{AW}$ as stated in formula~(\ref{e:AW}) is the ease for estimating the distance between two sets, as we can see in the following lemma. 
\begin{lemma} \label{aw}
Let $\left(X,d\right)$ be a metric space, $x_{0}\in X$ and  $\varepsilon\in (0,1)$. If $j\in \mathbb{N}$ is such that $\frac{1}{j+1}<\varepsilon\leq \frac{1}{j}$  then
\[
d_{AW}\left(A,B\right)<\varepsilon \mbox{ } \mbox{ if and only if } \mbox{ } \sup_{d\left(x,x_{0}\right)<j}\left|d\left(x,A\right)-d\left(x,B\right)\right|<\varepsilon,
\]
for all $A,B\in \textup{CL}(X)$, where $d_{AW}$ is based at the point $x_0$.
\end{lemma}

We omit the proof of this lemma, since it follows directly from the definition of $d_{AW}$.

 The Hausdorff metric topology and the Attouch-Wets metric topology on $\cl(X)$ are the topologies generated by the metrics $d_H$ and $d_{AW}$, respectively. These spaces are denoted by $\cl_H(X)$ and $\cl_{AW}(X)$.


It is well known that the Hausdorff topology and the Vietoris topology coincide on $\mathcal K(X)$ (where $X$ is a metric space). However, the relation between these two topologies can be expressed in terms of the lower and upper Vietoris topology.

In order to see this, first let us observe that  the Hausdorff topology is also the supremum of two weaker topologies: the \textit{lower Hausdorff topology} (generated by the 
 cuasipseudometric $d_{H^-}:\cl(X)\times\cl(X)\to[0,\infty]$ defined, for every $A, C\in \cl(X)$,  by
$d_{H^-}(A,C)=e(A,C)$) and the \textit{upper Hausdorff topology} (generated by the 
 cuasipseudometric
 $d_{H^+}:\cl(X)\times\cl(X)\to[0,\infty]$ defined, for every $A, C\in \cl(X)$,  by
$d_{H^+}(A,C)=e(C,A)$).

The reason why  these topologies  are called lower and upper, is justified in the following  proposition.

\begin{proposition}\label{p: contentions vietoris Hausdorff}
Let $(X,d)$ be a metric space. 
\begin{enumerate}[\rm(1)]
    \item The lower Vietoris topology on $\cl(X)$ is weaker than the lower Hausdorff topology. 
    \item The upper Hausdorff topology on $\cl(X)$ is weaker than the upper Vietoris topology.
    \item The lower Hausdorff topology on $\mathcal K(X)$ coincides with  the lower Vietoris topology.
    \item The upper Hausdorff topology on $\mathcal K(X)$ coincides with the upper Vietoris topology.
    \item The Hausdorff topology on $\mathcal K(X)$ coincides with the Vietoris topology. 
\end{enumerate}
\end{proposition}
For the proof of these properties we refer the reader to \cite[Lemma 2.4 and Theorem 2.5]{BERTHIAUME}
(see also \cite{Rodriguez-Romaguera}).

\section{Continuity of the induced function} \label{s:induced function}

Let $X$ and $Y$ be topological spaces. For every function  $f:X\to Y$, let us denote by $\widetilde f$ the \textit{induced function} $\widetilde{f}:\cl(X)\to\cl(Y)$ defined by
\begin{equation}\label{e:induced function}
    A\mapsto\overline{f(A)}.
\end{equation}

It is well known that if $f$ is continuous, $X$ and $Y$ are compact metric spaces, and if $\cl(X)$ and $\cl(Y)$ are endowed with the Hausdorff topology (which, in this case, coincides with the Fell and the Vietoris topologies), then the induced function $\widetilde{f}$ is always continuous (see for instance \cite{nadler}).

A similar situation occurs when we consider the hyperspaces $\mathcal K(X)$ and $\mathcal K(Y)$ endowed with the Vietoris topology (which, in this case,  coincides with the Hausdorff topology).

In the following proposition we summarize these and other  well-known facts regarding the continuity of the induced function.

\begin{proposition}\label{p:induced vietoris proposition}
Let $f:X\to Y$ be a function between topological spaces. 
\begin{enumerate}[\rm(1)]
    \item  $f$ is continuous if and only if  $\widetilde{f}:\cl_{V^-}(X)\to\cl_{V^-}(Y)$ is continuous.
 
     \item If $Y$ is a normal space and $f$ is continuous, then $\widetilde{f}:\cl_{V^+}(X)\to\cl_{V^+}(Y)$ is continuous.
     \item If $Y$ is a normal space and $f$ is continuous, then $\widetilde{f}:\cl_{V}(X)\to\cl_{V}(Y)$ is continuous.
  
     \item If $f$ is continuous then $\widetilde{f}:\mathcal{K}_{V^+}(X)\to\mathcal{K}_{V^+}(Y)$ is continuous.
     \item If $f$ is continuous then $\widetilde{f}:\mathcal{K}_{V}(X)\to\mathcal{K}_{V}(Y)$ is continuous.
\end{enumerate}
\end{proposition}

Even if it is not stated as we did, the reader can find a proof of these properties in \cite[Section 5]{michael}.


The continuity  of the induced function with respect to the Hausdorff metric was  explored in \cite[Theorem 3]{Hitchcock}, where  Hitchcock proves (in a more general context) that a map $f:X\to Y$ is uniformly continuous if and only if the induced map $\widetilde{f}:\cl(X)\to\cl(Y)$ is uniformly continuous with respect to the lower Hausdorff topology, the upper Hausdorff topology or the Hausdorff metric topology. He also characterizes the continuity of the induced function, in terms of certain properties of the map $f$ (\cite[Theorem 2]{Hitchcock}).

Before passing to the main results of this section, let us notice that in the case of  the  family of closed and bounded subsets of  a metric space,  the continuity of the induced function can be guaranteed by a weaker condition: uniform continuity on bounded sets.

Recall that a map $f:X\to Y$ between metric spaces $(X,d)$ and $(Y,d')$ is \textit{uniformly continuous on the  set $A\subset X$} if   for every $\varepsilon >0$, there exists $\delta>0$ such that if $a\in A$, $x\in X$ and $d(a,x)<\delta$, then $d'(f(a), f(x))<\varepsilon$ (\cite{beer di Concilio}). If $f$ is uniformly continuous on every bounded set $A\subset X$ we say that $f$ is \textit{uniformly continuous on bounded sets}. 

For a metric space $(X,d)$, let us denote by $BC(X)$ the family of all closed and bounded subsets of $X$.
At this point, the meaning of the  symbols 
$BC_{H}(X)$, $BC_{H^-}(X)$ and $BC_{H^+}(X)$ must be clear.


Observe that in general, the induced function $\widetilde f:BC(X)\to BC(Y)$ may not be well defined, since the image of a bounded set could be non bounded. In order to avoid this situation, we need  the underlying map $f:X\to Y$ to be boundedness preserving.  A  map $f:X\to Y$ between metric spaces is \textit{boundedness preserving} if $f(A)$ is bounded in $Y$ for every bounded set $A\subset X$.

\begin{proposition}\label{prop induced hausdorff} Let $(X,d)$ and $(Y,d')$ be metric spaces. If $f:X\to Y$ is a
 boundedness preserving and uniformly continuous on bounded sets map, then the following induced functions are  continuous:

\begin{enumerate}[\rm(1)]
    \item $\widetilde{f}:BC_{H^-}(X)\to BC_{H^-}(Y)$,
    \item $\widetilde{f}:BC_{H^+}(X)\to BC_{H^+}(Y)$,
    \item $\widetilde{f}:BC_H(X)\to BC_H(Y)$.
\end{enumerate}

\end{proposition}

\begin{proof}
(1) Let $ A\in BC(X)$ and $\varepsilon >0$. Since $f$ is uniformly continuous on the bounded set $A$, there exists $\delta>0$ such that $d'(f(a),f(b))<\varepsilon$ for every $a\in A$ and $b\in X$ satisfying $d(a,b)<\delta$.  Suppose that $d_{H^-}(A,B)<\delta$. Fix $a\in A$. Since $d_{H^-}(A,B)<\delta$, there exists $b\in B$ such that $d(a,b)<\delta$. Hence, $d'(f(a),f(B))\leq d'(f(a),f(b))<\varepsilon$. Since $a\in A$ was chosen arbitrarily, we can conclude that $$d'_{H^-}\big(\overline{f(A)},\overline{f(B)}\big)=d'_{H^-}(f(A),f(B))=\sup_{a\in A}d'(f(a),f(B))\le\varepsilon.$$

(2) Since the proof is similar to the previous one, we omit the details. 

(3) Since $d_H=\max\{d_{H^-},d_{H^+}\}$, this follows directly by combining  (1) and (2).
\end{proof}

In the following example we show that the hypothesis of $f$ being uniformly continuous on bounded sets is essential in Proposition \ref{prop induced hausdorff}.

\begin{example}
Let $X=(0,1)$ and $Y=[-1,1]$, both equipped with the Euclidean distance.
Consider the map $f:X\to Y$ given by $f(x)=\sin (1/x)$.
We claim that $\widetilde f: BC_H(X)\to BC_H(Y)$ is not continuous.
Indeed, consider the set $A=\{1/(2\pi n)\mid n\in\mathbb N\}$. Clearly $A$ is closed in $X$ and $\widetilde f(A)=\{0\}$.
Define  $A_k\subset X$ as
$$A_k:=A\cup \left [ \frac{1}{2\pi(k+1)},\frac{1}{2\pi k} \right].$$
Observe that $d_H(A_k, A)\to 0$ but $\widetilde f(A_k)=[-1,1]\not \to \widetilde f(A)$. Thus, $\widetilde f$ is not  continuous.

\end{example}

\subsection{Continuity of the induced function with respect to the Fell topology}

The continuity of the induced function with respect to the Fell topology has been studied in some specific contexts.  For instance, in \cite{wang} the authors  consider the case of a continuous map $f:X\to X$ where $X$ is a Hausdorff, locally compact and second countable space, and they proved that  the induced function $\widetilde f$ is continuous, provided that $f$ is perfect (namely, $f$ is closed and every fiber $f^{-1}(y)$ is compact). They also give conditions on the map $f$ in order that the continuity of the induced function implies that $f$ is perfect. 
For the proof of their results, the authors of \cite{wang} used strongly the hypothesis of $X$ being second countable (and hence, metrizable).

In the following theorem we present a more general scenario where the induced function $\widetilde f:\cl_{F}(X)\to\cl_{F}(Y)$  becomes continuous. Also, we characterize the continuity of $\widetilde f$ in the case that $f$ is a closed map.  

\begin{theorem}\label{t:induced fell}
Let $f:X\to Y$ be a continuous map between Hausdorff spaces.
\begin{enumerate}[\rm(1)]
    \item If $Y$ is locally compact and $f$ is proper, then $\widetilde{f}:\cl_{F^+}(X)\to\cl_{F^+}(Y)$ is continuous. In particular $\widetilde{f}:\cl_{F}(X)\to\cl_{F}(Y)$ is continuous.
    \item If $f$ is perfect, then $\widetilde{f}:\cl_{F^+}(X)\to\cl_{F^+}(Y)$ is continuous. In particular $\widetilde{f}:\cl_{F}(X)\to\cl_{F}(Y)$ is continuous.
    \item If $f$ is closed and non constant and $\widetilde{f}:\cl_{F}(X)\to\cl_{F}(Y)$ is continuous, then $f^{-1}(y)$ is compact for every $y\in f(X)$ (namely, $f$ is perfect).
    \item If $f$ is closed, $\widetilde{f}:\cl_{F}(X)\to\cl_{F}(Y)$ is continuous if and only if $f$ is perfect or constant.
\end{enumerate}
\end{theorem}

Recall that a map $f:X\to Y$ between topological spaces is called  \textit{proper} if for every compact subset $K$ of $Y$, the inverse image $f^{-1}(K)$ is a compact subset of $X$. We also recall that every perfect map is proper (\cite[Theorem 3.7.2]{engelking}).

\begin{proof}
 
 (1) Consider a compact subset $K\subset Y$ and an element $A\in \cl(X)$ such that 
$\widetilde{f}(A)\in (Y\setminus K)^+$.  Hence, $K\subset Y\setminus\overline{f(A)}$. Since $Y$ is a locally compact Hausdorff space, there exists an open subset $W\subset Y$ such that $\overline{W}$ is compact and $$K\subset W\subset\overline{W}\subset Y\setminus\overline{f(A)}.$$
Thus, using that $f$ is proper, we conclude that   $(X\setminus f^{-1}(\overline{W}))^+$ is an open neighborhood of $A$  and satisfies  $f(B)\subset Y\setminus W$, for every  $B\in (X\setminus f^{-1}(\overline{W}))^+$. Since $ W$ is open, we also have that $\overline{f(B)}\subset Y\setminus W\subset Y\setminus K$ for every  $B\in (X\setminus f^{-1}(\overline{W}))^+$ and therefore $\widetilde f$ is continuous. 

The last part follows from the previous argument in combination with  Proposition~\ref{p:induced vietoris proposition}-(1). 

(2) In this case, since $f$ is perfect, $f$ is closed and therefore $\widetilde f(A)=f(A)$ for every $A\in\cl(X)$. Thus, if $K\subset Y$ is compact, we have that

\begin{align*}
\widetilde{f}^{-1}\left(Y\setminus K\right)&=\{A\in \cl(X)\mid f(A)\subset Y\setminus K\}\\
&=\{A\in \cl(X)\mid A\subset f^{-1}(Y\setminus K)\}\\
&=\{A\in \cl(X)\mid A\subset X\setminus f^{-1}(K)\}\\
&=(X\setminus f^{-1}(K))^+.
\end{align*}
Since $f$ is perfect, then it is proper and therefore $(X\setminus f^{-1}(K))^+=\widetilde{f}^{-1}\left(Y\setminus K\right)$ is open in $\cl_{F^+}(X)$. Thus  $\widetilde{f}:\cl_{F^+}(X)\to\cl_{F^+}(Y)$ is continuous.  Using Proposition~\ref{p:induced vietoris proposition}-(1) again we also conclude that $\widetilde{f}:\cl_{F}(X)\to\cl_{F}(Y)$ is continuous.

(3) Let $y\in f(X)$ and consider the open set $(Y\setminus \{y\})^+$. By the continuity of $\widetilde f$ and the fact that $f$ is closed, we have that 
\[\mathcal O:=\widetilde{f}^{-1}\left((Y\setminus \{y\})^{+}\right)=\{A\in \cl(X)\mid A\subset X\setminus f^{-1}(y)\}
\]
is  open in $\cl_{F^+}(X)$. Since $f$ is non constant, there exist $x\in X$ such that $f(x)\neq y$. Hence $\{x\}\in \mathcal O$ and therefore we can find a compact set $K\subset X$ and open sets $V_1,\dots, V_n\subset X$ such that $$\{x\}\in \mathcal U:=(X\setminus K)^+\cap \left(\bigcap_{i=1}^nV_i^-\right)\subset \mathcal O.$$
 On the other hand, for any element $z\in f^{-1}(y)$ the set  $\{x,z\}\in\cl(X)$ but $\{x,z\}\notin \mathcal O$. In particular, $\{x,z\}\notin \mathcal U$. Since $x\in X\setminus K$ and $x\in V_{i}$ for every $i\in\{1,\dots, n\}$, we conclude that $z\notin X\setminus K$ and therefore $z\in K$. This implies that $f^{-1}(y)$ is a closed subset contained in  the compact set $K$, thus it is compact.

(4) The \textit{if} part follows directly from (2) and the fact that if $f$ is constant, then $\widetilde f$ is constant and therefore it is continuous. The \textit{only if} part is a direct consequence of (3).
\end{proof}

\subsection{Continuity of the induced function  with respect to the Attouch-Wets metric}\label{S:Attouch West}

To finish this section, we want to explore the problem of whether the induced function is continuous on the  hyperspace $\cl_{AW}(X):=(\cl(X), d_{AW})$.

Unlike the Hausdorff distance, the induced function  is not even continuous on  $\mathcal K_{AW}(X):=(\mathcal K(X), d_{AW})$, as we can see in the following example.

\begin{example}
\label{e:AW no continua}
For every non-bounded metric space $(X,d)$, the function  $f:X\to\mathbb R$  given by  
$$f(x)=\tan^{-1}(d(x_0,x)),\;\; x\in X.$$
is a $1$-Lispchitz map  such that the induced function $\widetilde f:\mathcal K_{AW}(X)\to\mathcal K_{AW}(\mathbb R)$ is not continuous, where $\mathbb R$ is equipped with the Euclidean metric.
\end{example}

The reader can verify the veracity of this example as an easy exercise. However, we want to point out that the reason why this example fails is hidden in the fact that the preimages of bounded sets are not necessarily bounded. This situation is better explained in the following proposition, which in turn provides an alternative proof for Example~\ref{e:AW no continua}.

\begin{proposition}\label{p:preimage are bounded}
Let $\left(X,d\right)$ and $\left(Y,d^{\prime}\right)$ be metric spaces and $f:X\rightarrow Y$ be a continuous function. If $\widetilde f:\mathcal{K}_{AW}(X)\rightarrow \mathcal{K}_{AW}(Y)$ is continuous, then $f^{-1}\left(B\right)$ is bounded for any bounded subset $B$ of $Y$ such that $\left|B\cap f\left(X\right)\right|>1$. 
\end{proposition} 

\begin{proof}

Let $x_{0}$ and $y_{0}$ be fixed points in $X$ and $Y$, respectively, so that the metrics $d_{AW}$ and $d^{\prime}_{AW}$ on $\mathcal{K}_{AW}(X)$ and $\mathcal{K}_{AW}(Y)$ are based at $x_0$ and $y_0$.

Let $B$ be a bounded subset of $Y$ such that $\left|B\cap f\left(X\right)\right|>1$. Then there exist $a,b\in f^{-1}\left(B\right)$ with $f\left(a\right)\neq f\left(b\right)$.   Since $B$ is bounded, we can find $M\in \mathbb{N}$ such that $B\subseteq B\left(y_{0}, M\right)$. 

Define $r:=d'(f(a), f(b))/2$ and $\varepsilon:=\min\left\{r,\frac{1}{M}\right\}$. Since $\widetilde f$ is continuous at $\left\{a\right\}$ and $\left\{b\right\}$, we can find $\delta\in (0,1)$ such that, for every  $C\in \mathcal{K}(X)$ and $c\in\{a,b\}$,
 \[
 d_{AW}\left(\left\{c\right\},C\right)<\delta
\; \Longrightarrow\;
d^{\prime}_{AW}\left(\left\{f\left(c\right)\right\},f\left(C\right)\right)<\varepsilon.
 \]

Finally, let $j,L\in \mathbb{N}$ be such that $\frac{1}{j+1}<\delta\leq \frac{1}{j}$ and $\{a,b\}\subset B\left(x_{0}, L\right)$. 

We claim that $f^{-1}\left(B\right)\subseteq B\left(x_{0},3k\right)$, where  $k:=\max\left\{j,L\right\}$. Indeed, if that is not the case,  there must exist $p\in f^{-1}\left(B\right)$ such that $d\left(x_{0},p\right)\geq 3k$. Take $x\in B\left(x_{0},k\right)$ an arbitrary point. Thus, $d(x,p)>2k$ and since $\{a,b\}\subset B\left(x_{0},L\right)\subseteq B\left(x_{0},k\right)$, we also have that $$d\left(x,c\right)\leq d\left(x,x_{0}\right)+d\left(x_{0},c\right)<2k$$
for $c\in\{a,b\}$.
This shows that  $d(x,a)<d(x,p)$ and $d(x,b)<d(x,p)$ for every  $x\in B\left(x_{0},k\right)$. Hence, if $c\in\{a,b\}$, we infer that
\begin{equation}\label{e:distancias cero}
  \sup_{d\left(x,x_{0}\right)<k}\left|d\left(x,\left\{c\right\}\right)-d\left(x,\left\{c,p\right\}\right)\right|=0.  
\end{equation}
Since $j\leq k$ and $\frac{1}{j+1}<\delta\leq \frac{1}{j}$ we can use  Lemma \ref{aw} in combination with equality~(\ref{e:distancias cero}) to conclude that 
$$d_{AW}\left(\left\{a\right\},\left\{a,p\right\}\right)<\delta\;\text{ and }\;d_{AW}\left(\left\{b\right\},\left\{b,p\right\}\right)<\delta.$$
By the choice of $\delta$,   $d^{\prime}_{AW}\left(\left\{f\left(c\right)\right\},\left\{f\left(c\right),f\left(p\right)\right\}\right)<\varepsilon$ for each  $c\in\{a,b\}$.
Hence, if $c\in\{a,b\}$,
\[
\min\left\{\frac{1}{M},\sup_{d^{\prime}\left(x,y_{0}\right)<M}\left|d^{\prime}\left(x,\left\{f\left(c\right)\right\}\right)-d^{\prime}\left(x,\left\{f\left(c\right),f\left(p\right)\right\}\right)\right|\right\}<\varepsilon,
\]
and since $\frac{1}{M}\geq \varepsilon$, we also have that 
\[
\left|d^{\prime}\left(x,\left\{f\left(c\right)\right\}\right)-d^{\prime}\left(x,\left\{f\left(c\right),f\left(p\right)\right\}\right)\right|<\varepsilon,
\]
for every $x\in B(y_0,M)$. In particular, since $p\in f^{-1}\left(B\right)$, then $f\left(p\right)\in B\subseteq B\left(y_{0},M\right)$ and
\[
d^{\prime}\left(f\left(p\right),f\left(c\right)\right)=\left|d^{\prime}\left(f\left(p\right),\left\{f\left(c\right)\right\}\right)-d^{\prime}\left(f\left(p\right),\left\{f\left(c\right),f\left(p\right)\right\}\right)\right|<\varepsilon\leq r.
\]
This implies   that $f\left(p\right)\in B\left(f\left(c\right),r\right)$,  for $c\in \{a,b\}$. Thus
$$d'(f(a), f(b))\leq d'(f(a), f(p))+d'(f(p)), f(b))<2r=d'(f(a), f(b)).$$
From this contradiction we conclude that  $f^{-1}\left(B\right)\subseteq B\left(x_{0}, 3k\right)$, proving that $f^{-1}\left(B\right)$ is bounded, as desired.
\end{proof}

Another necessary condition for the induced function to be continuous is the uniform continuity 
(on a certain family of subsets) of the original map. Before seeing this we need the following lemma.

\begin{lemma} \label{cerrados}
Let $\left(X,d\right)$ and $\left(Y,d^{\prime}\right)$ be metric spaces,  $f:X\rightarrow Y$ a continuous map and $\varepsilon>0$. Assume that $\left(x_{n}\right)_{n\in \mathbb{N}}$ and $\left(y_{n}\right)_{n\in \mathbb{N}}$ are sequences in $X$ such that $\left(d\left(x_{n},y_{n}\right)\right)_{n\in \mathbb{N}}$ converges to $0$ and with the property that $d^{\prime}\left(f\left(x_{n}\right),f\left(y_{n}\right)\right)\geq \varepsilon$ for all $n\in \mathbb{N}$. Then $\left\{x_{n} \mid n\in \mathbb{N}\right\}$ is a closed subset of $X$.
\end{lemma}
\begin{proof}
If $\left\{x_{n} \mid n\in \mathbb{N}\right\}$ is not closed in $X$, we can find $p\in X$ and  a subsequence $\left(x_{n_{k}}\right)$ converging to  $p$. Since $\left(d\left(x_{n},y_{n}\right)\right)$ converges to $0$, we also have that $\left(y_{n_{k}}\right)$ converges to $p$. Now, we can use the continuity of $f$ to infer that  $\left(f\left(x_{n_{k}}\right)\right)$ and $\left(f\left(y_{n_{k}}\right)\right)$ converge to $f\left(p\right)$. This contradicts the inequality $d^{\prime}\left(f\left(x_{n}\right),f\left(y_{n}\right)\right)\geq \varepsilon$  and therefore we can conclude that $\left\{x_{n} \mid n\in \mathbb{N}\right\}$ is closed in $X$.
\end{proof}

\begin{theorem} \label{t:continuidad uniforme}
Let $\left(X,d\right)$ and $\left(Y,d^{\prime}\right)$ be metric spaces and $f:X\rightarrow Y$ be a continuous function. If $\widetilde f:\cl_{AW}(X)\rightarrow \cl_{AW}(Y)$ is continuous, then $f$ is uniformly continuous on every $A\subseteq X$ such that $f\left(A\right)$ is bounded.
\end{theorem}

\begin{proof}
Let $x_{0}$ and $y_{0}$ be fixed points in $X$ and $Y$, respectively. We will assume that the metrics $d_{AW}$ and $d^{\prime}_{AW}$ on $\cl_{AW}(X)$ and $\cl_{AW}(Y)$ are based at $x_0$ and $y_0$.

Let $A$ be a subset of $X$ such that $f\left(A\right)$ is bounded and suppose that $f$ is not uniformly continuous on $A$. Then there must exist $\varepsilon>0$ and sequences $(a_n)_{n\in\mathbb N}\subset A$ and $(x_{n})_{n\in\mathbb N}\subset X$ such that $d\left(a_{n},x_{n}\right)<\frac{1}{n}$ and $d^{\prime}\left(f\left(a_{n}\right),f\left(x_{n}\right)\right)\geq \varepsilon$. By Efremovic's Lemma (\cite[3.3.1]{beer book}), we can assume that $d^{\prime}\left(f\left(a_{n}\right),f\left(x_{m}\right)\right)\geq \frac{\varepsilon}{4}$ for all $n,m\in \mathbb{N}$.
Therefore, Lemma \ref{cerrados} implies that   $B\vcentcolon=\left\{x_{n} \mid n\in \mathbb{N}\right\}\in\cl(X)$. We will prove the theorem by showing that $\widetilde f$  is not continuous at $B$.

Since $f\left(A\right)$ is bounded, we can find $j\in \mathbb{N}$ such that $f\left(A\right)\subseteq B\left(y_{0},j\right)$. Let $\varepsilon^{\prime}=\min\left\{\frac{\varepsilon}{4},\frac{1}{j}\right\}$. We claim that for all $\delta>0$ there exists $C\in \textup{CL}(X)$ such that $d_{AW}\left(B,C\right)<\delta$ but $d^{\prime}_{AW}\left(\widetilde f\left(B\right),\widetilde f\left(C\right)\right)\geq \varepsilon^{\prime}$. Indeed, let $\delta>0$ and $m\in \mathbb{N}$ with $\frac{1}{m}<\delta$. Consider
\[
C\vcentcolon=\left(B\backslash \left\{x_{m}\right\}\right)\cup \left\{a_{m}\right\}. 
\]
By Lemma \ref{cerrados}, $C\in\cl(X)$. Furthermore, for every $x\in X$ we have that
\[
\left|d\left(x,B\right)-d\left(x,C\right)\right|\leq d\left(x_{m},a_{m}\right)<\frac{1}{m}, 
\]
and therefore
\[
d_{AW}\left(B,C\right)=\sup_{k\in \mathbb{N}}\min\left\{\frac{1}{k},\sup_{d\left(x,x_{0}\right)<k}\left|d\left(x,B\right)-d\left(x,C\right)\right|\right\}\leq \frac{1}{m}<\delta.
\]

On the other hand, since $d^{\prime}\left(f\left(a_{m}\right),f\left(x_{n}\right)\right)\geq \frac{\varepsilon}{4}$ for every  $n\in \mathbb{N}$, we also have that $d^{\prime}\left(f\left(a_{m}\right),f\left(B\right)\right)\geq \frac{\varepsilon}{4}$. Furthermore, since $f\left(a_{m}\right)$ belongs to $f\left(A\right)\subseteq B\left(y_{0},j\right)$, we get
\[
\begin{split}
\sup_{d^{\prime}\left(x,y_{0}\right)<j}\left|d^{\prime}\left(x,f\left(B\right)\right)-d^{\prime}\left(x,f\left(C\right)\right)\right| & \geq \left|d^{\prime}\left(f\left(a_{m}\right),f\left(B\right)\right)-d^{\prime}\left(f\left(a_{m}\right),f\left(C\right)\right)\right| \\
& = d^{\prime}\left(f\left(a_{m}\right),f\left(B\right)\right)\geq \frac{\varepsilon}{4}\geq \varepsilon^{\prime}.
\end{split}
\]
Finally, since $\frac{1}{j}\geq \varepsilon^{\prime}$ we obtain that
\[ 
\begin{split}
d^{\prime}_{AW}\left(\overline{f\left(B\right)},\overline{f\left(C\right)}\right) & \geq \min\left\{\frac{1}{j},\sup_{d^{\prime}\left(x,y_{0}\right)<j}\left|d^{\prime}\left(x,\overline{f\left(B\right)}\right)-d^{\prime}\left(x,\overline{f\left(C\right)}\right)\right|\right\} \\
& = \min\left\{\frac{1}{j},\sup_{d^{\prime}\left(x,y_{0}\right)<j}\left|d^{\prime}\left(x,f\left(B\right)\right)-d^{\prime}\left(x,f\left(C\right)\right)\right|\right\}\geq \varepsilon^{\prime}.
\end{split}
\]
Therefore $\widetilde f$ is not continuous at $B$, a contradiction.
\end{proof}

Finally, we will characterize the continuity of the induced function precisely by the  conditions given in Proposition~\ref{p:preimage are bounded} and Theorem~\ref{t:continuidad uniforme}. In order to do that, let us prove the following technical lemma. 

\begin{lemma}\label{l:distancia igual a distancia en interseccion}
Let $(X,d)$ be a metric space, $x_0\in X$, $j>0$ and $C\subset X$ a nonempty set. If $L>2j+d(x_0, C)$,  then 
$$d(x,C)=d\big(x,C\cap B(x_0, L) \big)\;\text{ for every }\;x\in B\left(x_{0},j\right).$$
\end{lemma}

\begin{proof}
Let $x\in B\left(x_{0},j\right)$.  It suffices to show that there exists a point $z\in C\cap B\left(x_{0},L\right)$ with the property that $d\left(x,z\right)\leq d\left(x,y\right)$ for every $y\in C\backslash B\left(x_{0},L\right)$. Since
\[
d\left(x,C\right)\leq d\left(x,x_{0}\right)+d\left(x_{0},C\right)<j+d\left(x_{0},C\right),
\]
 we can pick $z\in C$ such that
 \begin{equation}\label{e:z}
     d\left(x,z\right)<j+d\left(x_{0},C\right).
 \end{equation}
 Therefore, by our hypothesis on $L$, we get that
\[
d\left(x_{0},z\right)\leq d\left(x_{0},x\right)+d\left(x,z\right)<j+j+d\left(x_{0},C\right)<L.
\]
This proves that  $z\in C\cap B\left(x_{0},L\right)$. Moreover, if $y\in C\backslash B\left(x_{0},L\right)$, then $d\left(x_{0},y\right)\geq L$. In this case inequality~(\ref{e:z}) guarantees that
\[
d\left(x,y\right)\geq d\left(x_{0},y\right)-d\left(x_{0},x\right)>L-j>j+d\left(x_{0},C\right)>d\left(x,z\right).
\]
Hence, $d\left(x,z\right)<d\left(x,y\right)$, as desired.
\end{proof}

\begin{theorem}\label{t:AWmain}
Let $\left(X,d\right)$ and $\left(Y,d^{\prime}\right)$ be metric spaces, and $f:X\rightarrow Y$ be a continuous function.  Then $\widetilde f:\cl_{AW}(X)\rightarrow \cl_{AW}(Y)$ is continuous if and only if the following two conditions are true:
\begin{enumerate}[\rm(1)]
\item $f$ is uniformly continuous on every subset $A$ of $X$ such that $f\left(A\right)$ is bounded.
\item The set $f^{-1}\left(B\right)$ is bounded for every bounded subset $B$ of $Y$ such that $\left|B\cap f\left(X\right)\right|>1$. 
\end{enumerate}
\end{theorem}

\begin{proof}
Let $x_{0}$ and $y_{0}$ be fixed points in $X$ and $Y$, respectively, and let us assume that $d_{AW}$ and $d^{\prime}_{AW}$ are based at $x_0$ and $y_0$, respectively.

If $\widetilde f:\cl_{AW}(X)\rightarrow \cl_{AW}(Y)$ is continuous, then Proposition~\ref{p:preimage are bounded} and Theorem~\ref{t:continuidad uniforme} guarantee that $(1)$ and $(2)$ are true.

For the converse implication, let us suppose that (1) and (2) hold. If $f$ is a constant function, then $\widetilde f$ is a constant function too and therefore it is continuous. Hence, we can suppose that $f$ is not constant. Let $A\in \cl(X)$, $\varepsilon\in (0,1]$ and  $j\in \mathbb{N}$  such that $\frac{1}{j+1}<\varepsilon\leq \frac{1}{j}$. Since $f$ is not constant, there exists $L\in \mathbb{N}$ with the property that $\left|B\left(y_{0},L\right)\cap f\left(X\right)\right|>1$ and
\begin{equation} \label{L}
L>2j+d^{\prime}\left(y_{0},f\left(A\right)\right)+\varepsilon.
\end{equation}

By (2),  $f^{-1}\left(B\left(y_{0},L\right)\right)$ is bounded and therefore there exists $M\in \mathbb{N}$ such that $f^{-1}\left(B\left(y_{0},L\right)\right)\subset B\left(x_{0},M\right)$.

Furthermore,  by (1) we can pick a  $\delta\in \left(0,\frac{1}{M}\right]$ such that for every $p\in f^{-1}\left(B\left(y_{0},L\right)\right)$ and $x\in X$, if $d\left(p,x\right)<\delta$ then $d^{\prime}\left(f\left(p\right),f\left(x\right)\right)<\varepsilon$.

Let $B\in \cl(X)$ with $d_{AW}\left(A,B\right)<\delta$. Since $\delta\leq \frac{1}{M}$, by Lemma \ref{aw} we have that
\begin{equation} \label{M}
\sup_{d\left(x,x_{0}\right)<M}\left|d\left(x,A\right)-d\left(x,B\right)\right|<\delta.
\end{equation}
In order to prove that  $\widetilde f$ is continuous at $A$ we need to show that  $d^{\prime}_{AW}\left(\widetilde f\left(A\right),\widetilde f\left(B\right)\right)<\varepsilon$.

Now, if we apply Lemma~\ref{l:distancia igual a distancia en interseccion} to the set $f(A)$ we obtain the following.
\begin{claim}
$d^{\prime}\left(x,f\left(A\right)\right)=d^{\prime}\left(x,f\left(A\right)\cap B\left(y_{0},L\right)\right)$ for every $x\in B\left(y_{0},j\right)$.
\end{claim}

\begin{claim}
$d^{\prime}\left(y_0,f\left(B\right)\right)\leq d^{\prime}\left(y_{0},f\left(A\right)\right)+\varepsilon.$
\end{claim}

\begin{proof}
 Consider $y\in f\left(A\right)\cap B\left(y_{0},L\right)$ and take $z\in A$ such that $f\left(z\right)=y$. Since $z$ belongs to $f^{-1}\left(B\left(y_{0},L\right)\right)\subseteq B\left(x_{0},M\right)$, by the inequality (\ref{M}) we have 
\[
d\left(z,B\right)=\left|d\left(z,A\right)-d\left(z,B\right)\right|<\delta.
\]
Hence, there exists $b\in B$ such that $d\left(z,b\right)<\delta$, and since $z\in f^{-1}\left(B\left(y_{0},L\right)\right)$, this implies that $d^{\prime}\left(f\left(z\right),f\left(b\right)\right)<\varepsilon$. Therefore,
\[
\begin{split}
d^{\prime}\left(y_0,f\left(B\right)\right)\leq d^{\prime}\left(y_0,f\left(b\right)\right) & \leq d^{\prime}\left(y_{0},f(z)\right)+d^{\prime}\left((f(z),f\left(b\right)\right) \\
& <d^{\prime}\left(y_{0},f(z)\right)+\varepsilon.
\end{split}
\]
Since $y=f(z)\in f\left(A\right)\cap B\left(y_{0},L\right)$ was chosen arbitrarily, we conclude that
$$d^{\prime}\left(y_0,f\left(B\right)\right)  \leq d^{\prime}\left(y_{0},f\left(A\right)\cap B\left(y_{0},L\right)\right)+\varepsilon. $$

Therefore Claim 2 follows from Claim 1.
\end{proof}

From Claim 2 we infer that $d'(y_0, f(B))+2j\leq d^{\prime}\left(y_{0},f\left(A\right)\right)+\varepsilon+2j$ and therefore we can use Lemma~\ref{l:distancia igual a distancia en interseccion} again to prove our last claim of the proof.

\begin{claim}
 $d^{\prime}\left(x,f\left(B\right)\right)=d^{\prime}\left(x,f\left(B\right)\cap B\left(y_{0},L\right)\right)$ for every $x\in B\left(y_{0},j\right)$.
\end{claim}

Let $z\in B$ be such that $f\left(z\right)=y\in f\left(B\right)\cap B\left(y_{0},L\right)$.  Since $z$ belongs to $f^{-1}\left(B\left(y_{0},L\right)\right)\subseteq B\left(x_{0},M\right)$, by the inequality (\ref{M}) we obtain
\[
d\left(z,A\right)=\left|d\left(z,A\right)-d\left(z,B\right)\right|<\delta.
\]
Hence, there exists  a point $a\in A$ such that $d\left(z,a\right)<\delta$, and since $z\in f^{-1}\left(B\left(y_{0},L\right)\right)$, this implies that $d^{\prime}\left(f\left(z\right),f\left(a\right)\right)<\varepsilon$. Therefore,
\begin{equation} \label{u}
d^{\prime}\left(y,f\left(A\right)\right)=d^{\prime}\left(f\left(z\right),f\left(A\right)\right)\leq d^{\prime}\left(f\left(z\right),f\left(a\right)\right)<\varepsilon.
\end{equation}

Let $x\in B\left(y_{0},j\right)$.
Then, by (\ref{u}),
$$d^{\prime}\left(x,f\left(A\right)\right)\leq d^{\prime}\left(x,y\right)+d^{\prime}\left(y,f\left(A\right)\right)<d^{\prime}\left(x,y\right)+\varepsilon.$$ 
Since this is true for every $y\in f\left(B\right)\cap B\left(y_{0},L\right)$, we conclude that
\[
d^{\prime}\left(x,f\left(A\right)\right)\leq d^{\prime}\left(x,f\left(B\right)\cap B\left(y_{0},L\right)\right)+\varepsilon.
\]
Hence, by Claim 3, we obtain that  $d^{\prime}\left(x,f\left(A\right)\right)-d^{\prime}\left(x,f\left(B\right)\right)\leq \varepsilon$. 
Analogously we can prove that $d^{\prime}\left(x,f\left(B\right)\right)-d^{\prime}\left(x,f\left(A\right)\right)\leq \varepsilon$, and therefore $\left|d^{\prime}\left(x,f\left(A\right)\right)-d^{\prime}\left(x,f\left(B\right)\right)\right|\leq \varepsilon$. Since this is true for every $x\in B\left(y_{0},j\right)$, we conclude that
\[
\sup_{d\left(x,y_{0}\right)<j}\left|d^{\prime}\left(x,\overline{f\left(A\right)}\right)-d^{\prime}\left(x,\overline{f\left(B\right)}\right)\right|=
\]
\[
\sup_{d\left(x,y_{0}\right)<j}\left|d^{\prime}\left(x,f\left(A\right)\right)-d^{\prime}\left(x,f\left(B\right)\right)\right|\leq\varepsilon.
\]

Finally, we can use that $\frac{1}{j+1}<\varepsilon\leq \frac{1}{j}$ in combination with Lemma \ref{aw}, to conclude that $d^{\prime}_{AW}\left(\overline{f\left(A\right)},\overline{f\left(B\right)}\right)\leq\varepsilon$. This proves that $\widetilde f$ is continuous at every $A\in \cl(X)$.
\end{proof}

\section{Continuity of the induced action}\label{s:acciones}

Through this section we use the letter $G$ to denote a topological group. The given topology of the group is denoted by $\tau_G$. If there exists a continuous action of $G$ on a topological space $X$,  we say that $X$ is a $G$-space.  As usual, the image of the pair $(g,x)\in G\times X$  under the action of $G$ on $X$ will be simply denoted by $gx$. We refer the reader to the monograph \cite{Bredon}  for a deeper understanding of the theory of $G$-spaces.

If $X$ is a $G$-space,  every element of the group $g\in G$ defines a homeomorphism $X\to X$ by means of the rule $x\mapsto gx$. To avoid extra notation, we denote this homeomorphism by $g$. 
This implies that  $g(A)$ is closed for every $A\in \cl(X)$ and every $g\in G$. Therefore,  the action of  $G$ on $X$  induces a natural (algebraic) action on $\cl(X)$ defined by the following rule:
 $$(g,A)\mapsto gA=\{ga : a\in A\}.$$
  This action is  called {\it the natural action} of $G$ on  $\cl(X)$, or simply {\it the induced action} of $G$ on $\cl(X)$.
  Even if this action is well defined, it may fail to be continuous. 
  
In this section  we approach the following problem: if $X$ is a $G$-space, under what conditions does the hyperspace $\cl(X)$ become a $G$-space?  In our results, we put special interest in the topology of the group $G$. We will also consider the case of the hyperspace of all closed and bounded subsets equipped with the Hausdorff metric.

To approach this question, first 
let us denote by $\ho (X)$ the group of all homeomorphisms of a topological space $X$. If $X$ is a $G$-space, then there exists an obvious group monomorphism $G\hookrightarrow \ho (X)$. Thus,  the group $G$ can be considered as an algebraic subgroup of $\ho (X)$.

If $A,B\subset X$, let  $$[A,B]=\{f\in G :f(A)\subset B\}.$$

Using this notation, let us recall that the \textit{compact-open topology} on $G$ is the topology generated by all sets of the form $[K,V]$, where $K$ is compact and $V$ is open in $X$. 
On the other hand, the \textit{closed-open topology} on $G$ is the topology generated by all sets of the form $[C,V]$, where $C$ is closed and $V$ is open in $X$.

In \cite{arens}, Arens studied which function-space-topologies turn $\ho(X)$ into a topological group. In particular, 
if $X$ is a normal space and $\ho(X)$ is equipped with the closed-open topology,  then $\ho(X)$ is always a topological group. 
On the other hand, if  $\ho(X)$ is endowed with the compact-open topology and $X$ is locally connected and locally compact, then $\ho(X)$ is a topological group. The same remains true if $X$ is just a compact space. 
 
However, being a topological group is not enough to guarantee the continuity of the action on $X$, nor on $\cl(X)$. In \cite{concilio}, the author studied the problem of whether $\cl(X)$ is a $\ho(X)$-space, where both $\cl(X)$ and $\ho(X)$ are endowed with certain topologies that depend on a so-called Urysohn family in $X$. From  \cite[Theorem 4.1]{concilio}, we can conclude that for a locally compact space $X$, the group $\ho(X)$ endowed with the compact-open topology is a  topological group if and only if 
the evaluation map $E:\ho(X)\times \cl_{F}(X)\to\cl_F(X)$ is continuous, provided that  $\ho(X)$ is equipped with the compact-open topology.  This result was previously proved by R. K. Wicks in the unpublished paper \cite{Wicks}. 

In a similar way, from \cite[Theorem 4.1]{concilio} we can also infer that  $\ho(X)$ endowed with the closed-open topology being a topological group is equivalent to the continuity of the evaluation map $E:\ho(X)\times \cl_V(X)\to\cl_V(X)$, provided that $X$ is a normal space and $\ho(X)$ is equipped wit the closed open topology.

Following this line, we want to start by making some remarks when we consider the Fell topology and the Vietoris topology on $\cl(X)$. Even if these remarks
 cannot be  directly infered from the sentence of Di Concilio's Theorem,  the technique we use to prove them is very similar to the one used in \cite{concilio}.

\begin{remark}{(c.f. \cite[Theorem 4.1-5]{concilio})}\label{r:coninuity action lower vietoris}
Let $X$ be  $G$-space. Then $\cl_{V^-}(X)$ equipped with the induced action of $G$ is always a $G$-space.
\end{remark}

\begin{proof}
Let $(g,A)\in G\times\cl_{V^-}(X)$ be such that $gA\in O^-$, for some open set $O\subset X$. Then, there exists a point  $a\in A$ such that $ga\in O$. Since the action of $G$ is continuous on $X$, we can pick two open neighbourhoods $U\subset G$ and $W\subset X$ of $g$ and $a$, respectively, such that $hy\in O$ for every $h\in U$ and $y\in W$.
Hence, for every pair $(h, B)\in U\times W^{-}$, there must exist a point $b\in B\cap W$ and therefore $hb\in O$. This implies that $hB\in O^{-}$ and therefore the induced action is continuous.
\end{proof}

After this remark, we conclude that the continuity of the induced actions  on the hyperspaces $\cl_F(X)$ and $\cl_V(X)$ depends on the upper Fell topology and upper Vietoris topology, respectively.

 Now, let us recall that if $X$ is a $G$-space, then the continuity of the action on $X$ guarantees that the topology of $G$ must be admissible and therefore it must contain the compact-open topology. 
Recall that a topology $\tau$ on a  subgroup $G\leq \ho(X)$ is called admissible if the evaluation map $E:G\times X\to X$ is continuous (see \cite[Section 3]{arens}).

 \begin{remark}
\label{r: contiene compacto abierta}
Let $(G,\tau_G)$ be a topological group and $X$ be a $G$-space. Then $ \tau_G$ contains the compact-open topology of $G$.
 \end{remark}

The proof of this lemma is well-known and will be omitted.


Let $A$ and $B$ be arbitrary subsets of a $G$-space $X$. Observe that an element of the group $g\in G$ lies in $[A, B]$ if and only if $g^{-1}$ lies in $[X\setminus B, X\setminus A]$. Since the inversion of the group is a homeomorphism of $G$, the set $[A, B]$ is open in $G$ if and only if $[X\setminus B, X\setminus A]$ is open in $G$. This fact in combination with Remark~\ref{r: contiene compacto abierta}
implies the following remark.

\begin{remark}\label{r:abierto compacto abierta}
Let $X$ be a $G$-space.
\begin{enumerate}
\item $[A,B]\subset G$ is open in $G$ if and only if $[X\setminus B, X\setminus A]$ is open in $G$. 
    \item If $K\subset X$ is compact and $U\subset X$ is open, then the set $[X\setminus U,X\setminus K]$ is open in $G$. 
\end{enumerate}
\end{remark}

\begin{remark}
\label{r: action Fell topology}
Let $X$ be a locally compact $G$-space. Then the induced action of $G$ on $\cl_{F^+}(X)$ is continuous. In particular, $\cl_{F}(X)$ is a $G$-space.
\end{remark}

\begin{proof} By Remark~\ref{r:coninuity action lower vietoris} it is enough to prove that the induced action of $G$ on $\cl_{F^{+}}(X)$ is continuous.
Let $K\subset X$ be a compact subset and consider a pair  $(g,A)\in G\times \cl_{F^+}(X)$  be such that $gA\in (X\setminus K)^+$.  Thus  $gA\subset X\setminus K$ and therefore  $g^{-1}K\subset X\setminus A$. Since $X$ is locally compact, there exists an open set $V\subset X$ such that $\overline{V}$ is compact and $$g^{-1}(K)\subset V\subset \overline{V}\subset X\setminus A,$$

then 
$$gA\subset g(X\setminus\overline{V})\subset g(X\setminus V)\subset X\setminus K.$$

Let $O=[X\setminus V, X\setminus K]\times (X\setminus \overline{V})^+$. By Remark~ \ref{r:abierto compacto abierta}, $O$ is an open set and contains the pair $(g,A)$. If $(h,B)\in O$, clearly
$$hB\subset h(X\setminus \overline{V})\subset h(X\setminus V)\subset X\setminus K.$$
This proves that $hB\in (X\setminus K)^+$ and therefore the action on $\cl_{F^+}(X)$ is continuous.

\end{proof}

Observe that every element of the group  $G$ determines a perfect bijection on $X$. Then Theorem~\ref{t:induced fell}-(3) guarantees that each element of $G$ induces a continuous  function on $\cl_F(X)$, regardless of $X$ is locally compact or not. However, the local compactness of $X$ is essential in Remark~\ref{r: action Fell topology}, as we can see in the following example.

\begin{example}
Let $G=\mathbb{Q}*$ be the multiplicative group of the field $\mathbb{Q}$ and $X=\mathbb{Q}$. Clearly, $G$ acts continuously on $X$. We will show that the induced action on $\cl_{F^+}(X)$ is not continuous. Consider the compact set $K=\{1\}\cup\{1+1/n:n\in\mathbb{N}\}$ and let $A$ be any closed set in $X$ such that $A\subset X\setminus K$. We will show that the action on $\cl_{F^+}(X)$ is not continuous at the point $(1,A)$.  

Let $\varepsilon>0$ and consider any compact set  $C\subset X$  such that $A\subset X\setminus C$. Consider a positive scalar $\delta >0$ with the property that  $1-\varepsilon<1/(1+\delta)$.
We will show that $(X\setminus K)^+$ cannot contain the image under the action of $G$ of the set $(1-\varepsilon,1+\varepsilon)\times(X\setminus C)^+$.

Indeed, since $C$ is compact, it cannot contain the set $(1,1+\delta)\cap\mathbb{Q}$. Therefore, there exists $v\in (1,1+\delta)\cap\mathbb{Q}$ such that $v\in X\setminus C$. Let $n\in\mathbb{N}$ be small enough such that $1/n<\varepsilon /2$ and let $r=v^{-1}(1+1/n)$. Then $r\in (1-\varepsilon,1+\varepsilon)$, $\{v\}\in (X\setminus C)^+$ and $rv=1+1/n\in K$.  
This shows that the induced action is not continuous. 
\end{example}

\begin{remark}
 Let $(G,\tau_G)$ be a topological group and
 $X$ be a $G$-space. Then:
\begin{enumerate}[\rm(1)]
    \item If the induced action of $G$ on  $\cl_{V^+}(X)$ is continuous, then   $\tau_G$ contains the closed-open topology of $G$.
    \item If $\tau_G$ contains the closed-open topology of $G$ and $X$ is a normal space, then the induced action of $G$ on  $\cl_{V^+}(X)$ is continuous.
    \item If $\tau_G$ contains the closed-open topology of $G$ and $X$ is a normal space, then the induced action of $G$ on  $\cl_{V}(X)$ is continuous.
\end{enumerate}
\end{remark}
 
 \begin{proof}
 (1) Let $[C,U]$ be a subbasic open set in the closed-open topology of $G$ and let $g\in [C,U]$. Thus,  $g(C)\in U^+$. Since the induced action of $G$ on $\cl_{V^+}(X)$ is continuous, there exist open sets $W\subset G$ and $O\subset X$ such that $g\in W$, $C\in O^+$ and $WO^+\subset U^+$. This implies that $W\subset [C,U]$, and therefore $[C,U]\in \tau_{G}$.

  We omit the proof of (2) and (3), since it is quite similar to the one of Remark~\ref{r: action Fell topology}.
 \end{proof}

\subsection{Group actions on hyperspaces equipped with the Hausdorff metric topology}

For a metric $G$-space $X$, it is well-known and easy to verify that the induced action on $(\mathcal K(X), d_H)$ is always continuous (see \cite[Section 3]{natalia sergey}).

Since the Hausdorff metric is not very well behaved on $\cl(X)$, it is not surprising to find that the induced action of $G$ on $\cl(X)$ is not necessarily continuous. Consider, for example, the induced action of the group $\mathbb S^1$ on $\cl(\mathbb C)$. In this case, every neighbourhood $U\subset\mathbb S^1$ of the identity contains an element $w\neq 1$. If we consider the closed set $A=\{z \in\mathbb C : \text{Re}(z)\geq 0,\; \text{Im}(z)=0\}$, it is not difficult to verify that $d_H(A, wA)=\infty$ and therefore the natural action of $\mathbb S^1$ on $\cl(\mathbb C)$ cannot be continuous.

This is why  we first focus our attention on the hyperspace $BC(X)$ of closed and bounded subsets of $X$. However, even in  $BC(X)$ the induced action may fail to be continuous. Not even if the group consists of isometries. Indeed, in   \cite[Example 3.1]{Jonard 2014},  there is an example of a $G$-Banach space $X$, where $G$ is a compact group of isometries, such that the induced action of $G$ on $BC(X)$ is not continuous.

Let $(X,d)$ be a metric space and $G\leq \ho (X)$ be any subgroup. Recall that the  \textit{topology of the uniform convergence on bounded sets} ($\tau_{u.c.b.}$) on $G$ is the topology generated by all sets of the form
    $$(A,f,\varepsilon)=\{g\in G:d(f(x),g(x))<\varepsilon\;\forall x\in A\},$$ where $A\subset X$ is bounded, $f\in G$ and $\varepsilon>0$.
 Using this notation, let us also recall that the \textit{topology of the uniform convergence} on $G$ is 
the topology generated by the sets of the form $(X,f,\varepsilon)$, where $f\in G$ and $\varepsilon>0$.

The following is a simple generalization of \cite[Example 3.3]{Jonard 2014}.

\begin{theorem}\label{t:G en BCH}
Let $(G, \tau_G)$ be a topological group and  $X$ be a metric $G$-space. Assume that the homeomorphisms on $X$ induced by every $g\in G$ are boundedness preserving and uniformly continuous on bounded sets.
If $\tau_G$ contains the topology of the uniform convergence on bounded sets,  then the induced actions of $G$ on $BC_{H^-}(X)$, $BC_{H^+}(X)$ and $BC_H(X)$ are continuous.
\end{theorem}

\begin{proof}
Let $(g,A)\in G\times BC(X)$ and $\varepsilon >0$. Since $g$ is uniformly continuous on $N_{\varepsilon/2}(A)$, there exists $\delta>0$ such that if $a,a'\in N_{\varepsilon/2}(A)$ and $d(a,a')<\delta$ then $d(ga,ga')<\varepsilon/2$.

Consider $\varepsilon'=\min\{\varepsilon/2,\delta\}$ and let $$W=(N_{\varepsilon /2}(A),g,\varepsilon/2)\times B_{H^-}(A,\varepsilon').$$

Clearly, $W$ is a neighborhood of $(g,A)$ in the product space $G\times BC_{H^-}(X)$. Let $(h,B)\in W$ and $a\in A$. Since  $d_{H^-}(A, B)<\varepsilon'$, there exists $b\in B$ such that $d(a,b)<\varepsilon'$, so $b\in N_{\varepsilon/2}(A)$ and $d(a,b)<\delta$. Therefore $d(ga,gb)<\varepsilon /2$. Furthermore, since $b\in N_{\varepsilon/2}(A)$, by the definition of $W$ we have that $d(gb, hb)<\varepsilon/2$. Thus
$$d(ga,hb)\le d(ga,gb)+d(gb,hb)<\varepsilon/2+\varepsilon/2=\varepsilon.$$

This proves that $d_{H^-}(gA,hB)=\sup_{a\in A}d(ga,hB)<\varepsilon$, and hence the induced action of $G$ on $BC_{H^-}(X)$ is continuous.

In a similar way, we can prove that such action is continuous on $BC_{H^+}(X)$, and therefore it is continuous on $BC_H(X)$ as well.\end{proof}

Observe that the hypothesis  that every element of the group $G$ is boundedness preserving  (regarded as a homeomorphism from $X$ onto $X$),   is  a
necessary condition to guarantee that the action on $BC(X)$ is well defined.   On the other hand, the hypothesis of $g$ being uniformly continuous on bounded sets also implies that the induced map $\widetilde g$ is continuous (Proposition~\ref{prop induced hausdorff}). 

Finally, on the induced action on the hyperspaces $\cl_{H^-}(X)$, $\cl_{H^+}(X)$ and $\cl_H(X)$, we have the following result.

\begin{theorem}
\label{t: G en clH}
Let $(G, \tau_G)$ be a topological group and $X$ be a metric $G$-space. Assume that the homeomorphisms on $X$ induced by every $g\in G$ are uniformly continuous on $X$. If $\tau_G$ contains the topology of the uniform convergence, then the induced actions of $G$ on $\cl_{H^-}(X)$, $\cl_{H^+}(X)$ and $\cl_H(X)$ are continuous.
\end{theorem}

Since the proof of Theorem~\ref{t: G en clH} is very similar to the one of Theorem~\ref{t:G en BCH}, we leave it as an easy exercise to the reader.

\medskip

\subsection{Group actions on  hyperspaces with the Attouch-Wets metric}

Let $(X,d)$ be a metric $G$-space. If we want to guarantee that each element   $g\in G$ induces a continuous function on $\cl_{AW}(X)$, the homeomorphism associated to $g$ must satisfy conditions (1) and (2) of Theorem~\ref{t:AWmain}. However, since each $g\in G$ determines a bijection from $X$ onto $X$, condition (2) holds if and only if every $g\in G$ is boundedness preserving. This in turn implies that condition (1) holds if and only if each element of $G$ is uniformly continuous on bounded sets.
In the last theorem we will see that these two conditions, in combination with the fact that $\tau_G$ contains the topology of uniform convergence on bounded sets, are sufficient to guarantee that the induced action on $\cl_{AW}(X)$ is continuous.

\begin{theorem}\label{t: G on AW}
Let $G$ be a topological group and  $(X,d)$  a metric $G$-space with the following properties.
\begin{enumerate}[\rm(1)]
    \item Each $g\in G$ determines a homeomorphism on $X$ which is boundedness preserving and uniformly continuous on bounded sets.
    \item The topology of the group $G$ contains the topology of the uniform convergence on bounded sets. 
\end{enumerate}
Then the induced action of $G$ on $\cl_{AW}(X)$ is continuous.
\end{theorem}

\begin{proof}
Let us assume that the Attouch-Wets metric on $\cl(X)$ is based at the point $x_0\in X$.
Let $(g,A)\in G\times \cl(X)$ and $\varepsilon\in (0,1)$. Take $j\in\mathbb N$ and $L>0$ such that $\frac{1}{j+1}<\varepsilon/2\leq\frac{1}{j}$ and
\begin{equation}\label{e:L accion}
    L>d(x_0, gA)+2j+\varepsilon.
\end{equation}

Define $Q:=N_{\varepsilon/4}\left(g^{-1}B(x_0,L)\right)$. Since $g^{-1}$ is boundedness preserving, the set $Q$ is bounded and therefore $ O_1=\left(Q,g,\frac{\varepsilon}{4}\right)$ is a neighborhood of $g$ in  $\tau_{G}$.

On the other hand, for every $\varphi\in O_2:=\left(B(x_0, L), g^{-1}, \varepsilon/4 \right)$, the set $\varphi B(x_0,L)$ is  completely contained in $Q$.
Let $O:=O_1\cap O_2^{-1}$. By the continuity of the inversion on $G$, the set $O$ is a neighborhood of $g$ in $\tau_G$.
Since $g$ induces a continuous map on $\cl_{AW}(X)$ (Theorem~\ref{t:AWmain}), we can pick $\delta>0$ such that
$d_{AW}(gA,gB)<\varepsilon/4$
for every $B\in\cl(X)$ such that $d_{AW}(A,B)<\delta$.

Take $h\in O$ and $B\in \cl(X)$ with $d_{AW}(A, B)<\delta$. Using Lemma~\ref{aw} it is easy to verify that $d_{AW}(gA, gB)<\varepsilon/4$ implies that
\begin{equation}\label{e:distancia entre gB y gA}
   d(x_0, gB)-d(x_0, gA)<\varepsilon/4. 
\end{equation}
Thus $d(x_0, gB)+2j< d(x_0, gA)+\varepsilon/4+2j<L$. Hence we can apply Lemma~\ref{l:distancia igual a distancia en interseccion} to conclude that
\begin{equation}\label{e:gB}
    d(x, gB)=d\big(x, gB\cap B(x_0, L)\big)\;\text{ for every }x\in B(x_0, j).
\end{equation}
Pick  $gb\in gB\cap B(x_0, L)$ with the property that
$$d(x_0, gb)<d\big(x_0, gB\cap B(x_0,L)\big)+\varepsilon/4=d(x_0, gB)+\varepsilon/4.$$
Since $gb\in B(x_0, L)$, we infer that $b\in Q$. Thus,
$d(gb, hb)<\varepsilon/4$ and therefore, by inequality~(\ref{e:distancia entre gB y gA}),
\begin{align*}
 d(x_0, hB)&\leq d(x_0, hb)\leq d(x_0, gb)+d(gb, hb)\\
 &< d(x_0, gB)+\varepsilon/4+\varepsilon/4
 = d(x_0, gB)+\varepsilon/2\\
 &\leq  d(x_0, gA)+\varepsilon/4+\varepsilon/2< d(x_0, gA)+\varepsilon.
\end{align*}
We infer from the previous inequality that $d(x_0, hB)+2j<d(x_0, gA)+\varepsilon+2j<L$ and hence we can apply Lemma~\ref{l:distancia igual a distancia en interseccion} again to obtain
\begin{equation}\label{e:hB}
  d(x, hB)=d\big(x, hB\cap B(x_0, L)\big)\;\text{ for every }x\in B(x_0, j).
\end{equation}

Let $x\in B(x_0, j)$ and $hy\in hB\cap B(x_0, L)$.
Then $y\in B\cap Q$, $gy\in gB$ and $d(gy, hy)<\varepsilon/4$. Therefore
\begin{align*}
    d(x,gB)&\leq d(x,hy)+d(hy, gB)\\
    &\leq d(x,hy)+d(hy,gy)\\
    &< d(x,hy)+\varepsilon/4.
\end{align*}
Since $hy\in hB\cap B(x_0, L)$ was taken arbitrarily, by equality~\ref{e:hB} we conclude that
$$d(x,gB)\leq d\big(x, hB\cap B(x_0, L)\big)+\varepsilon/4=d(x, hB)+\varepsilon/4.$$
Thus $d(x,gB)-d(x,hB)\leq\varepsilon/4$. Analogously we can prove that $d(x,hB)-d(x,gB)\leq\varepsilon/4$ and therefore
$$|d(x,gB)-d(x,hB)|\leq \varepsilon/4<\varepsilon/2\;\text{ for every }x\in B(x_0,j).$$
The previous inequality in combination with Lemma~\ref{aw} implies that 
$d_{AW}(gB, hB)<\varepsilon/2$ and therefore
$$d_{AW}(gA, hB)\leq d_{AW}(gA, gB)+d_{AW}(gB, hB)<\varepsilon/4+\varepsilon/2<\varepsilon.$$
This completes the proof. 

 \end{proof} 
\subsection{Final questions}

In \cite[Theorem 4.2]{concilio}, Di Concilio characterizes whether $\ho(X)$, equipped with the topology of the uniform convergence on a certain family of sets,  becomes  a topological group. In the particular case of the  topology of the uniform convergence on bounded sets,  \cite[Theorem 4.2]{concilio} implies that for a metric space $(X,d)$, the topology $\tau_{u.c.b.}$  on $\ho(X)$ is a group topology if and only if the evaluation map $E:\ho(X)\times\cl(X)\to\cl(X)$ is continuous, provided that $\ho(X)$ is endowed with   $\tau_{u.c.b.}$, and $\cl(X)$ is equipped with  the hit and far miss topology on $\cl(X)$ generated by $\tau_{V^-}$ and the sets of the form
$$(X\setminus B)^{++}=\{A\in\cl(X)\mid \exists \varepsilon>0\text{ such that }N_{\varepsilon}(A)\subset X\setminus B\}$$
where $B\in BC(X)$.
Since, in general,  this topology does not coincide with the Hausdorff metric topology nor the Attouch-Wets metric topology, it is natural to ask if the condition on theorems~\ref{t:G en BCH} and \ref{t: G on AW} of $\tau_{G}$ being stronger than $\tau_{u.c.b}$ is necessary. Thus, we finish the paper with the following two questions.

\begin{question}
Let $X$ be a metric $G$, where each element of $G$ is uniformly continuous on bounded sets and boundedness preserving. If the induced action of $G$  on $BC_{H}(X)$ (or $\cl_{AW}(X)$) is continuous, does $\tau_G$ contain the topology of the uniform convergence on bounded sets?
\end{question}

\begin{question}
Let $X$ be a metric $G$-space, where each element of $G$ is uniformly continuous. If the induced action of $G$  on $\cl_{H}(X)$  is continuous, does $\tau_G$ contain the topology of the uniform convergence?
\end{question}

\begin{center}
ACKNOWLEDGMENTS
\end{center}

We wish to  thank  the anonymous  referee for the constructive comments and recommendations
which definitely  improved the  quality of this paper.

\end{document}